\documentclass[11pt, reqno]{amsart}

\usepackage[left=2.5cm, right=2.5cm, top=2cm, bottom=2cm]{geometry}

\usepackage[utf8]{inputenc}

\usepackage{amssymb, amsmath, amsthm, esint, enumitem,dsfont,tikz}
\usetikzlibrary{angles,quotes}

\usepackage{latexsym}

\usepackage{graphicx}

\usepackage{hyperref}

\usepackage{mathtools}
\mathtoolsset{showonlyrefs}

\usepackage{color}

\newtheorem{defi}{Definition}[section]
\newtheorem{thm}[defi]{Theorem}
\newtheorem{ex}[defi]{Example}
\newtheorem{rem}[defi]{Remark}

\newtheorem{lemma}[defi]{Lemma}

\newcommand{\R}{\mathbb R}

\newcommand{\s}{\mathbb{S}}

\title[Spectral ratios and gaps for Steklov eigenvalues]{Spectral ratios and gaps for Steklov eigenvalues of balls with revolution-type metrics}
\author{Jade Brisson}
\address{Institut de Math\'ematiques, Universit\'e de Neuch\^atel, Rue Emile-Argand 11, 2000 Neuch\^atel, Suisse}
\email{jade.brisson@unine.ch}
\author{Bruno Colbois}
\address{Institut de Math\'ematiques, Universit\'e de Neuch\^atel, Rue Emile-Argand 11, 2000 Neuch\^atel, Suisse}
\email{bruno.colbois@unine.ch}
\author{Katie Gittins}
\address{Department of Mathematical Sciences, Durham University,
Mathematical Sciences and Computer Science Building,
Upper Mountjoy Campus, Stockton Road,
Durham DH1 3LE,
United Kingdom.}
\email{katie.gittins@durham.ac.uk}
\date{\today}

\begin{document}
\begin{abstract}
We investigate upper bounds for the spectral ratios and gaps for the Steklov eigenvalues of balls with revolution-type metrics. We do not impose conditions on the Ricci curvature or on the convexity of the boundary. We obtain optimal upper bounds for the Steklov spectral ratios in dimensions 3 and higher. In dimension 3, we also obtain optimal upper bounds for the Steklov spectral gaps. By imposing additional constraints on the metric, we obtain upper bounds for the Steklov spectral gaps in dimensions 4 and higher.
\end{abstract}
\maketitle

\section{Introduction}

The spectral ratio of the first two Dirichlet eigenvalues $\lambda_2^D/\lambda_1^D$ on Euclidean domains has received a great deal of attention since the work of Payne, Pólya and Weinberger \cite{PPW55, PPW56}. They conjectured that in 2 dimensions, the best constant for $\lambda_2^D/\lambda_1^D$ is that achieved by the disc and this conjecture was generalised to higher dimensions by Thompson \cite{T69}. Both conjectures were proven by Ashbaugh and Benguria in \cite{AB92}. In addition, it was shown by Andrews and Clutterbuck \cite{AC11}, respectively Payne and Weinberger \cite{PW60}, that among all convex domains in $\mathbb{R}^n$ of prescribed diameter, the gap between the first two Dirichlet, respectively Neumann, eigenvalues is minimised by the line segment. Results for the spectral ratio and spectral gap for the Robin eigenvalues on rectangles under various geometric constraints have been obtained by Laugesen \cite{L19}. These results lend support to conjectures in broader classes of domains and we refer the reader to \cite{L19} and references therein for further details. It is known that the spectral ratio for consecutive non-trivial Neumann eigenvalues $\lambda_{k+1}^N/\lambda_k^N$ can be arbitrarily large. For example, by taking $k$ disjoint balls in $\mathbb{R}^n$ and joining them by thin cylinders we obtain a Cheeger dumbbell whose first $k$ Neumann eigenvalues are very small but the $(k+1)$-st one is not (see, for example, \cite[Example 18]{BC17}).

In this paper, we investigate the spectral ratios of the Steklov eigenvalues. Our results also shed light on the Steklov spectral gaps. 

It is well known that the Steklov eigenvalues of a smooth, compact, connected Riemannian manifold $(M,g)$ of dimension $n\geq 2$ with boundary $\Sigma$ are the real numbers $\sigma$ for which there exists a nonzero harmonic function $u:M\rightarrow\R$ which satisfies $\partial_\nu u=\sigma u$ on the boundary $\Sigma$. Throughout this paper, $\partial_\nu$ is the outward-pointing normal derivative on $\Sigma$. 
We will denote the Steklov eigenvalues as
$$ 0=\sigma_0<\sigma_{1}\leq\sigma_{2}\leq\cdots\nearrow\infty,$$
where each eigenvalue is repeated according to its multiplicity.

In general, reminiscent of the situation for the Neumann spectral ratio mentioned above, the Steklov spectral ratio can be arbitrarily large.
For example, in \cite[Section 2.1]{GP10}, a family of bounded, simply-connected domains in $\mathbb{R}^2$, $\Omega_\epsilon$, is constructed such that the Steklov spectral ratio $\sigma_{k+1}/\sigma_{k}$ of the limiting domain as $\epsilon \to 0$ is arbitrarily large. Indeed, for $k \in \mathbb{N}$, $\Omega_\epsilon$ consists of $k+1$ overlapping discs each of unit radius such that as $\epsilon \to 0$, $\Omega_\epsilon$ degenerates to the disjoint union of $k+1$ discs, $B_1, \dots, B_{k+1}$, each of unit radius and the authors show that $\lim_{\epsilon \to 0} \sigma_{k+1}(\Omega_\epsilon) |\partial \Omega_\epsilon| = 2\pi (k+1)$. But, it is possible to show that  $\lim_{\epsilon \to 0} \sigma_{k}(\Omega_\epsilon) = 0$. In addition, the generalisation of this example to two overlapping domains in higher dimensions has been addressed in \cite[Example 3]{BGT20}.
In a similar spirit, it is also possible to obtain examples of Riemannian manifolds for which the Steklov spectral ratio can be made arbitrarily large by employing Theorem 1.1 of \cite{FS20} which asserts that for a collection $M_1, \dots, M_s$ of compact $n$-dimensional Riemannian manifolds with non-empty boundary and $\epsilon >0$, there exists a Riemannian manifold $M_\epsilon$ that is obtained by gluing $M_1, \dots, M_s$ suitably along their boundaries such that for $k = 0, 1, 2, \dots$, $\lim_{\epsilon \to 0} \sigma_k(M_\epsilon) = \sigma_k(M_1 \sqcup \dots \sqcup M_s)$. So taking $s = k+1$, we see that $\sigma_k(M_1 \sqcup \dots \sqcup M_{k+1}) = 0$ while $\sigma_{k+1}(M_1 \sqcup \dots \sqcup M_{k+1}) > 0$.
Therefore, in order to obtain bounds on the Steklov spectral ratio, it is necessary to impose additional geometric constraints. 

To that end, in this paper, we consider the case where $M= [0,R] \times \mathbb{S}^{n-1}$ and $$ g := g_h = dr^2 + h(r)^2 g_0,$$
where $g_0$ is the usual metric on the $(n-1)$-dimensional unit sphere and $h(R)=0$ so that $(M,g)$ corresponds to a metric of revolution on the ball. 
In this setting, we denote the Steklov eigenvalues counted without multiplicities as
$$ 0=\sigma_0(g_h) <\sigma_{(1)}(g_h) < \sigma_{(2)}(g_h) < \cdots\nearrow\infty.$$

In \cite{X21,X22}\footnote{Note that we use the convention that $r=0$ corresponds to the boundary whereas the convention in \cite{X21, X22} is that $r=R$ corresponds to the boundary which explains the differences in the statement of those results here.} the author considers this setting with the additional assumptions that $(M,g)$ has strictly convex boundary and either non-negative Ricci curvature or non-positive Ricci curvature.
In the case where the Ricci curvature of $(M,g)$ is non-negative, the author shows in \cite{X22} that
\begin{equation}\label{eqx1}
    \sigma_{(k)}(g_h) \geq k\frac{(-h'(0))}{h(0)}, \quad k \geq 0,
\end{equation}
with equality if and only if $h(r) = R - r$ or $M$ is isometric to the Euclidean ball of radius $R$.  The author also obtains a corresponding result when the Ricci curvature of $(M,g)$ is non-positive where the inequality in \eqref{eqx1} is reversed.
The case where the Ricci curvature is non-negative and $k=1$ is related to a conjecture of Escobar about a lower bound for the first non-trivial Steklov eigenvalue (see \cite{X22} and references therein). For an overview of lower bounds for the first non-trivial Steklov eigenvalue see \cite[Section 4.1]{CGGS2024}.

In \cite{X21}, the author investigates the spectral gaps and ratios of the Steklov eigenvalues.
For the case where the Ricci curvature of $(M,g)$ is non-negative, the author shows in \cite{X21} that
\begin{itemize}
    \item when $n=2$,
    \begin{equation}\label{eqx2}
        \sigma_{(k)}(g_h) = \frac{k}{h(0)}, \quad k \geq 0,
    \end{equation}
    \item when $n \geq 3$, 
    \begin{equation}\label{eqx3}
        \sigma_{(k+1)}(g_h) - \sigma_{(k)}(g_h) \geq \frac{-h'(0)}{h(0)}, \quad k \geq 0,
    \end{equation}
    and
    \begin{equation}\label{eqx4}
        \frac{\sigma_{(k+1)}(g_h)}{\sigma_{(k)}(g_h)} \leq \frac{k+1}{k}, \quad k \geq 1,
    \end{equation}
\end{itemize}
with equality in \eqref{eqx3} or \eqref{eqx4} if and only if $h(r) = R - r$ or $M$ is isometric to the Euclidean ball of radius $R$. The author also obtains corresponding results when the Ricci curvature of $(M,g)$ is non-positive where the inequalities in \eqref{eqx3}, \eqref{eqx4} are reversed.

In this paper, we obtain optimal upper bounds for the Steklov spectral ratios
\begin{equation*}
    \frac{\sigma_{(k+1)}(g_h)}{\sigma_{(k)}(g_h)}
\end{equation*}
when $n \geq 3$ and the Steklov spectral gaps
\begin{equation*}
    \sigma_{(k+1)}(g_h) - \sigma_{(k)}(g_h)
\end{equation*}
when $n=3$ without any assumptions on the curvature of $(M,g)$ or any convexity assumptions on the boundary. 
By imposing additional assumptions on the metric $h$, we also obtain upper bounds for the Steklov spectral gaps when $n \geq 4$.

Throughout, analogously to \cite{X21, X22}, we impose the following constraints on $h$ which ensure that the metric is smooth.
\begin{itemize}
    \item [(H)] $h \in C^{\infty}([0,R])$, $h(r) > 0$ for $r \in [0,R)$, $h'(R) = -1$ and $h^{(2k)}(R) = 0$ for all $k \in \mathbb{Z}, k \geq 0$.
\end{itemize}
In addition, we denote the eigenvalues of the Laplacian on the $(n-1)$-dimensional unit sphere $\mathbb{S}^{n-1}$ with usual metric $g_0$ counted without multiplicity by $\lambda_{(k)}$. As in \cite{X21,X22}, we note that $\sigma_{(k)}$ has the same multiplicity as $\lambda_{(k)}$.
Our main results are the following. 

\begin{thm}\label{thm1}
    Let $M= [0,R] \times \mathbb{S}^{n-1}$ be equipped with the metric $g_h = dr^2 + h(r)^2 g_0,$ where $g_0$ is the usual metric on the $(n-1)$-dimensional unit sphere and $h$ satisfies assumptions $(H)$.
    For $k \geq 1$ we have that
    \begin{itemize}
        \item when $n \geq 3$,
        \begin{equation}\label{eqthm1}
        \frac{\sigma_{(k+1)}(g_h)}{\sigma_{(k)}(g_h)} < \frac{\lambda_{(k+1)}}{\lambda_{(k)}}
        = \frac{(k+1)}{k}\frac{(n+k-1)}{(n+k-2)},
    \end{equation}
        \item when $n=2$,
        $$\sigma_{(k)}(g_h) =\frac{k}{h(0)}.$$
    \end{itemize}
\end{thm}
\noindent
Moreover, the upper bound \eqref{eqthm1} is optimal.
\begin{thm}\label{thm2}
Let $n \geq 3$. Let $M= [0,R] \times \mathbb{S}^{n-1}$ be equipped with the metric $g_h = dr^2 + h(r)^2 g_0,$ where $g_0$ is the usual metric on the $(n-1)$-dimensional unit sphere and $h$ satisfies assumptions $(H)$.
For $k \geq 1$ we have that
\begin{equation}
    \sup_h \frac{\sigma_{(k+1)}(g_h)}{\sigma_{(k)}(g_h)} = \frac{\lambda_{(k+1)}}{\lambda_{(k)}}
\end{equation}
where the supremum is taken over all $h$ satisfying $(H)$.
\end{thm}

 We prove Theorem \ref{thm2} by constructing a suitable family of metrics that are very large on a substantial part of $[0,R]$.
 However, it is surprising that this is not the only construction that ensures the Steklov spectral ratio approaches the supremum. We explore a different construction where the metrics become very small and have this property in Example \ref{eg1} for $n \geq 4$.
 These explorations shed light on some constraints that can be imposed on the metric $h$ so that the Steklov spectral ratio is not close to the supremum. More precisely, we prove the following theorem in Section \ref{ss:appsup}.

\begin{thm}\label{thm3}
    Let $n \geq 3$. Let $M= [0,R] \times \mathbb{S}^{n-1}$ be equipped with the metric $g_h = dr^2 + h(r)^2 g_0,$ where $g_0$ is the usual metric on the $(n-1)$-dimensional unit sphere and $h$ satisfies assumptions $(H)$.
    Suppose there exist $C_2 > C_1 > 0$ and $0 < R_1 < R$ such that
    \begin{equation}\label{3.10a}
        h(r) \leq C_2, \quad \mbox{ for } 0 \leq r \leq R,
    \end{equation}
    and 
    \begin{equation}\label{3.10b}
        h(r) \geq C_1, \quad \mbox{ for } 0 \leq r \leq R_1.
    \end{equation}
    Then 
    \begin{equation}\label{3.11}
        \frac{\sigma_{(k+1)}(g_h)}{\sigma_{(k)}(g_h)} \leq \frac{\lambda_{(k+1)}}{\lambda_{(k)}} - \gamma,
    \end{equation}
    with
 \begin{equation}
    \begin{split}
        \gamma &= \min\left\{\frac{1}{4R_1} \frac{C_1^{2(n-1)}}{C_2^{2(n-2)}}(\lambda_{(k+1)} - \lambda_{(k)})\left( R \lambda_{(k)}^2 + \frac{C_2^2}{R- R_1}\lambda_{(k)}\right)^{-1}, \right. \\ 
        & \left. \quad \quad \quad \frac{C_1^{4(n-2)}}{C_2^{2(2n-3)}} \frac{R_1^3}{128} (\lambda_{(k+1)} - \lambda_{(k)}) \left( R + \frac{C_2^2}{(R- R_1)\lambda_{(k)}}\right)^{-1}\right\}.
        \label{eq:gamma}
    \end{split}
    \end{equation}
\end{thm}

\begin{rem}
    We observe that when $C_2 \to \infty$ or when $R_1 \to 0$, the right-hand side of \eqref{eq:gamma} tends to 0. These cases correspond to constructions that ensure that the Steklov spectral ratio approaches the supremum which we explore in Section \ref{ss:thm1.2}.
    We note that the case $R_1 \to R$ is not possible as $h(R) = 0$ (see (H)).
\end{rem}

In addition, when $n \geq 4$, for the family of metrics constructed in the proof of Theorem \ref{thm2} (respectively Example \ref{eg1})  the Steklov spectral gap is arbitrarily large (respectively small).
However, when $n = 3$, it is possible to obtain the following upper bound for the Steklov spectral gap.
\begin{thm} \label{thm:gaprevolution}
    Let $M = \lbrack0,R\rbrack\times\s^{2}$ be equipped with the metric $g_h=dt^2+h(t)^2g_{\s^{2}}$, where $g_{\s^{2}}$ is the usual metric on $\s^2$ and $h$ satisfies assumptions $(H)$. For each $k \geq 0$, we have
    \[\sigma_{(k+1)}(g_h)-\sigma_{(k)}(g_h) < \frac{R(\lambda_{(k+1)}-\lambda_{(k)})}{h(0)^2}\,.\]
    Moreover, if we fix the value of $h$ at $t=0$, the upper bound is optimal:
    \[\sup\limits_h\{\sigma_{(k+1)}(g_h)-\sigma_{(k)}(g_h):h(0)=h_0\}=\frac{R(\lambda_{(k+1)}-\lambda_{(k)})}{h_0^2}\,.\]
\end{thm}
In order to prove Theorem \ref{thm:gaprevolution}, we make use of the following result which is an upper bound for the Steklov eigenvalues in this setting.
\begin{thm}\label{thm:revolution}
    Let $M = \lbrack0,R\rbrack\times\s^{2}$ be equipped with the metric $g_h=dt^2+h(t)^2g_{\s^{2}}$, where $g_{\s^{2}}$ is the usual metric on $\s^2$ and $h$ satisfies assumptions $(H)$. For $k \geq 1$, we have that
    \[\sigma_{(k)}(g_h)<\frac{R\lambda_{(k)}}{h(0)^2}\,.\]
    Moreover, if we fix the value of $h$ at $t=0$, the bound is sharp. Namely, we have
    \[\sup\limits_h\{\sigma_{(k)}(g_h):h(0)=h_0\}=\frac{R\lambda_{(k)}}{h_0^2}=\frac{Lk(k+1)}{h_0^2}\,.\]
\end{thm}

For the case where $n \geq 4$, under bounds on $h$ as in Theorem \ref{thm3}, it is also possible to obtain upper bounds for the Steklov spectral gap as given in the next theorem.

\begin{thm}\label{lem:gap}
Let $n \geq 4$ and $M= [0,R] \times \mathbb{S}^{n-1}$ be equipped with the metric $g_h = dr^2 + h(r)^2 g_0,$ where $g_0$ is the usual metric on the $(n-1)$-dimensional unit sphere and $h$ satisfies assumptions $(H)$. If $h(r) \leq C_2$ for $0 \leq r \leq R$, then
    \begin{equation}\label{eq:gap1}
        \sigma_{(k+1)}(g_h) - \sigma_{(k)}(g_h) \leq \frac{(\lambda_{(k+1)}-\lambda_{(k)}) C_2^{n-3} R}{h(0)^{n-1}}.
    \end{equation}
\end{thm}
\noindent

\subsection*{Plan of the paper}
In Section \ref{ss:tools} we recall some tools in this geometric setting that will be used in the proofs of our results. We then prove Theorem \ref{thm1} in Section \ref{ss:thm1.1}, Theorem \ref{thm2} in Section \ref{ss:thm1.2} and Theorem \ref{thm3} in Section \ref{ss:appsup}. The proof of Theorem \ref{thm:revolution} is given in Section \ref{ss:upbound3D} and Section \ref{ss:gap} contains the proof of Theorem \ref{thm:gaprevolution} and that of Theorem \ref{lem:gap}.

\subsection*{Acknowledgements}
J.B and B.C acknowledge support of the SNSF project ‘Geometric Spectral Theory’, grant number 200021-19689.
We are grateful to Alexandre Girouard for informing us of reference \cite{BGT20}.
We are also very grateful to the referee for many helpful suggestions and corrections.


\section{Tools in this geometric setting}\label{ss:tools}
In this section we recall some well known facts about the Steklov eigenvalue problem on manifolds with revolution-type metrics (see, for example, \cite[Proposition 11]{X21}).

If $\{\varphi_j\}_{j=0}^\infty$ is an $L^2$-orthonormal basis of eigenfunctions of the Laplacian on $\mathbb{S}^{n-1}$, i.e.
\begin{equation*}
    -\Delta \varphi_j = \lambda_j \varphi_j, \quad \lambda_j = j(n-2+j),
\end{equation*}
then the eigenfunctions of $(M,g_h)$ have the form
$ a_j \varphi_j$ where $a_j : [0,R] \to \mathbb{R}$ is a non-trivial solution of
\begin{equation}\label{eq:defa}
\begin{cases}
    \frac{1}{h^{n-1}} \frac{d}{dr}\left(h^{n-1}\frac{d}{dr}a_j\right) - \frac{\lambda_ja_j}{h^2} = 0, \quad r \in (0,R),\\
    a(R)=0.
\end{cases}
\end{equation}

We observe that the Rayleigh quotient of $a_j \varphi_j$ is
\begin{equation}\label{eq:rq}
    \mathcal{R}(a_j\varphi_j) = 
    \frac{\int_0^R \{(a_j')^2h^{n-1} + \lambda_j a_j^2h^{n-3}\} \, dr}{a_j(0)^2 h(0)^{n-1} }.
\end{equation}
Hence, if $\lambda_{(k)}$ is the $k$-th eigenvalue of $\mathbb{S}^{n-1}$ counted without multiplicity, then we have that
\begin{equation*}
    \sigma_{(k)}(g_h) = \min_{a:[0,R] \to \mathbb{R}, a(R)=0} 
    \frac{\int_0^R \{(a')^2h^{n-1} + \lambda_{(k)} a^2h^{n-3}\} \, dr}{a(0)^2 h(0)^{n-1}}.
\end{equation*}
By comparing the Rayleigh quotient \eqref{eq:rq} and the results of Theorem \ref{thm1} and Theorem \ref{thm2}, we observe that  in order to obtain these results, the term involving $a_j'$ must vanish. This observation will be key to the strategies of the proofs that follow.

The following classic result will also be useful in the arguments that follow.

\begin{lemma}\label{lem:fundcs}
    Let $a : \mathbb{R} \to \mathbb{R}$ be differentiable. Then, for $\alpha, \beta \in \mathbb{R}$, we have
    \begin{equation*}
        |a(\beta) - a(\alpha)|^2 \leq |\beta - \alpha| \int_\alpha^\beta a'(r)^2 \, dr. 
    \end{equation*}
\end{lemma}

\begin{proof}
    By the Fundamental Theorem of Calculus, we have that
    \begin{equation*}
        |a(\beta) - a(\alpha)| = \bigg\vert \int_\alpha^\beta a'(r) \, dr\bigg\vert. 
    \end{equation*}
    Applying the Cauchy--Schwarz inequality to the right-hand side gives
    \begin{equation*}
        \bigg\vert \int_\alpha^\beta a'(r) \, dr\bigg\vert \leq \sqrt{|\beta - \alpha|}\left(\int_\alpha^\beta a'(r)^2\,dr\right)^{1/2}
    \end{equation*}
    from which we deduce the required result by squaring.
\end{proof}

\section{Proofs of main results}
In this section we give the proofs of our main results.

\subsection{Proof of Theorem \ref{thm1}}\label{ss:thm1.1}

\begin{proof}[Proof of Theorem \ref{thm1}]
    We first consider the case where $n \geq 3$. In order to find an upper bound for $$\frac{\sigma_{(k+1)}(g_h)}{\sigma_{(k)}(g_h)}, $$
    we take a function $a_k$ that gives rise to an eigenfunction for $\sigma_{(k)}(g_h)$, that is
    \begin{equation*}
    \sigma_{(k)}(g_h) = 
    \frac{\int_0^R \{(a_k')^2h^{n-1} + \lambda_{(k)} a_k^2h^{n-3}\} \, dr}{a_k(0)^2 h(0)^{n-1}},
\end{equation*}
    and use it as a test function in the Rayleigh quotient corresponding to $\sigma_{(k+1)}(g_h)$.
    We have the following
    \begin{align*}
    \sigma_{(k+1)}(g_h) \leq \mathcal{R}(a_k\varphi_{k+1}) &= 
    \frac{\int_0^R \{(a_k')^2h^{n-1} + \lambda_{(k+1)} a_k^2h^{n-3}\} \, dr}{a_k(0)^2 h(0)^{n-1}}\\
    &=\frac{\int_0^R \{(a_k')^2h^{n-1} + \lambda_{(k)} a_k^2h^{n-3}\} \, dr}{a_k(0)^2 h(0)^{n-1}}
    + \frac{\int_0^R \{\lambda_{(k+1)}-\lambda_{(k)}\} a_k^2h^{n-3} \, dr}{a_k(0)^2 h(0)^{n-1}}.
    \end{align*}
    So
    \begin{align}
        \sigma_{(k+1)}(g_h)
        & \leq \sigma_{(k)}(g_h) + \frac{\lambda_{(k+1)}-\lambda_{(k)}}{\lambda_{(k)}}\frac{\int_0^R \lambda_{(k)} a_k^2h^{n-3} \, dr}{a_k(0)^2 h(0)^{n-1}} \nonumber \\
        & \leq \sigma_{(k)}(g_h) + \frac{\lambda_{(k+1)}-\lambda_{(k)}}{\lambda_{(k)}}\sigma_{(k)}(g_h). \label{eq:gap0}
    \end{align}
    Hence we deduce that
    \begin{equation*}
        \frac{\sigma_{(k+1)}(g_h)}{\sigma_{(k)}(g_h)}
        \leq \frac{\lambda_{(k+1)}}{\lambda_{(k)}}.
    \end{equation*}
    In order to have
    \begin{equation*}
        \frac{\sigma_{(k+1)}(g_h)}{\sigma_{(k)}(g_h)}
        = \frac{\lambda_{(k+1)}}{\lambda_{(k)}},
    \end{equation*}
    we must have equality in Inequality \ref{eq:gap0}. In particular,
    \begin{equation*}
        \frac{\int_0^R \lambda_{(k)} a_k^2h^{n-3} \, dr}{a_k(0)^2 h(0)^{n-1}} = \sigma_{(k)}(g_h)
    \end{equation*}
    which implies that 
    \begin{equation*}
        \frac{\int_0^R (a_k')^2h^{n-1} \, dr}{a_k(0)^2 h(0)^{n-1}} = 0
    \end{equation*}
    and hence $a_k'(r) = 0$ for almost every $r \in [0,R]$. However, this would give that $a_k$ is a constant function which is not possible as we know $a_k(R) = 0$ but the $a_k$ are non-trivial. Alternatively, constant functions do not satisfy the ODE in \eqref{eq:defa}. Therefore we conclude that
    \begin{equation*}
        \frac{\sigma_{(k+1)}(g_h)}{\sigma_{(k)}(g_h)}
        < \frac{\lambda_{(k+1)}}{\lambda_{(k)}}.
    \end{equation*}

    Finally, we consider the case where $n=2$. If $g(r,\theta)=dr^2+h(r)^2d\theta^2$ is a Riemannian metric on the disc $D$, then the length of the boundary of $(D,g)$ is $2\pi h(0)$. Via a homothety of ratio $\frac{1}{h(0)}$, $(D,g)$ is conformal to $(D,g_0)$, with boundary of length $2\pi$. Moreover, $\sigma_{(k)}(D,g)=\frac{1}{h(0)}\sigma_k(D,g_0)$.
    Now, as in \cite[Prop. 1.10]{CoGiGi2019}, $(D,g_0)$ is conformal to the Euclidean unit disc, with a conformal factor taking the value $1$ on the boundary. This implies that the Steklov spectrum of $(D,g_0)$ is the same as the Steklov spectrum of the unit Euclidean disc and $\sigma_{(k)}(D,g)=\frac{k}{h(0)}$.
\end{proof}

\subsection{Proof of Theorem \ref{thm2}}\label{ss:thm1.2} 

The key idea of the proof of Theorem \ref{thm2} is to choose a sequence of functions $(h_\epsilon)_\epsilon$, $0 < \epsilon < 1$, such that when $\epsilon \to 0$, the supremum of 
\begin{equation*}
    \frac{\int_0^R \{(a')^2h_\epsilon^{n-1} + \lambda a^2h_\epsilon^{n-3}\} \, dr}{a(0)^2 h(0)^{n-1}}
\end{equation*}
is given by 
\begin{equation*}
    \frac{\int_0^R  \lambda a^2h_\epsilon^{n-3} \, dr}{a(0)^2 h(0)^{n-1}}.
\end{equation*}
To achieve this, we choose $h_\epsilon$ so that they become very large on a substantial part of $[0,R]$ and we show that this leads to $a$ being close to a constant. 

However, using such a family of functions $h_\epsilon$ is not the only way to approach the supremum and we explore another possible family in Example \ref{eg1} for which the functions become very small.

\begin{proof}[Proof of Theorem \ref{thm2}]
    We first prove Theorem \ref{thm2} for $n \geq 4$. 
    For $\epsilon$ sufficiently small, we consider the following function:
    \begin{equation}\label{2.1}
        \tilde{h}_\epsilon (r)
        = \begin{cases}
            1, & r \leq \epsilon,\\
            \epsilon^{-1/2(n-3)}, & 2\epsilon\leq r \leq R-2\epsilon,\\
            R-r & R-\epsilon \leq r \leq R.
        \end{cases}
    \end{equation}
    We then define $h_\epsilon : [0,R] \to \mathbb{R}$ to be the function that is smooth, increasing on $[\epsilon,2\epsilon]$, decreasing on $[R-2\epsilon,R-\epsilon]$ and equal to $\tilde{h}_\epsilon$ otherwise. We observe that $h_\epsilon$ satisfies assumptions (H). 

    For $k \geq 1$, we are interested in the following quantity
    \begin{equation*}
        \mathcal{R}_k(a) = \frac{\int_0^R \{(a')^2h_\epsilon^{n-1} + \lambda_{(k)} a^2h_\epsilon^{n-3}\} \, dr}{a(0)^2}.
    \end{equation*}
    Without loss of generality, we suppose that $a(0) =1$. We observe that taking
    \begin{equation}\label{2.2}
        \tilde{a}(r) = 
        \begin{cases}
            1, & r \leq R - \epsilon,\\
            \frac{R-r}{\epsilon}, & R-\epsilon \leq r \leq R,
        \end{cases}
    \end{equation}
    as a test function gives the following upper bound for the Rayleigh quotient
    \begin{align}
        \mathcal{R}_k(\tilde{a}) & \leq
        \int_{0}^{R-\epsilon} \lambda_{(k)} \left(\frac{1}{\epsilon}\right)^{1/2} \, dr 
         + \int_{R-\epsilon}^R \left(\frac{1}{\epsilon}\right)^{2} (R-r)^{n-1} \, dr \\
        & \quad + \int_{R-\epsilon}^R \lambda_{(k)} (R-r)^{n-3} \frac{(R-r)^2}{\epsilon^2}  \, dr \\
        & =  \int_{0}^{R-\epsilon} \lambda_{(k)} \left(\frac{1}{\epsilon}\right)^{1/2} \, dr + (1 + \lambda_{(k)}) \int_{R-\epsilon}^R \left(\frac{1}{\epsilon}\right)^{2} (R-r)^{n-1} \, dr \\
        & = \left(\lambda_{(k)} \left(\frac{1}{\epsilon}\right)^{1/2}  (R-\epsilon) + (1+\lambda_{(k)}) \frac{\epsilon^{n-2}}{n} \right) \label{2.8}.
    \end{align}
    Hence, we have that
    \begin{equation}\label{2.3}
        \sigma_{(k)}(g_{h_\epsilon}) \leq \mathcal{R}_k(\tilde{a}) \leq \left( \lambda_{(k)} \left(\frac{1}{\epsilon}\right)^{1/2}  (R-\epsilon) + (1+\lambda_{(k)}) \frac{\epsilon^{n-2}}{n} \right).
    \end{equation}

    If instead, $a$ gives rise to an eigenfunction for $\sigma_{(k)}(g_{h_\epsilon})$ then we have that
    \begin{equation}\label{2.4}
        \sigma_{(k)}(g_{h_\epsilon}) \geq \int_0^{2\epsilon} (a')^2 h_\epsilon^{n-1} \, dr 
        \geq \int_0^{2\epsilon} (a')^2 \, dr 
        \geq \frac{1}{2\epsilon} |1 - a(2\epsilon)|^2
    \end{equation}
    by Lemma \ref{lem:fundcs}.
    Hence, from Inequalities \eqref{2.3} and \eqref{2.4} we deduce that
    \begin{equation}\label{2.5}
        |1 - a(2\epsilon)|^2 \leq 2\epsilon\left(\frac{\lambda_{(k)}}{ \epsilon^{1/2}}(R-\epsilon) + (1+\lambda_{(k)})\frac{\epsilon^{n-2}}{n}\right)
    \end{equation}
    which implies that
    \begin{equation}\label{2.6}
        a(2\epsilon) = 1 + O(\epsilon^{1/4}).
    \end{equation}
    When $a$ gives rise to an eigenfunction for $\sigma_{(k)}(g_{h_\epsilon})$, we also have that
    \begin{equation}\label{2.7}
        \mathcal{R}_k(a) \geq \int_{2\epsilon}^{R-2\epsilon} \left\{(a')^2 \left(\frac{1}{\epsilon}\right)^{(n-1)/2(n-3)} +\lambda_{(k)} a^2 \left(\frac{1}{\epsilon}\right)^{1/2}\right\} \, dr.
    \end{equation}
    Then, by combining Inequality \eqref{2.7} and Inequality \eqref{2.8}, we deduce that
    \begin{align}
        \int_{2\epsilon}^{R-2\epsilon} (a')^2 \left(\frac{1}{\epsilon}\right)^{1/(n-3)} \left(\frac{1}{\epsilon}\right)^{1/2} \, dr
        &\leq \left( \lambda_{(k)} \left(\frac{1}{\epsilon}\right)^{1/2}  (R-\epsilon) + (1+\lambda_{(k)}) \frac{\epsilon^{n-2}}{n} \right) \\
        & = \left(\lambda_{(k)} \left(\frac{1}{\epsilon}\right)^{1/2}  (R-\epsilon) +  C\epsilon^{n-2} \right), 
    \end{align}
    where $C = \frac{(1+\lambda_{(k)})}{n}$, which implies that
    \begin{equation}
        \int_{2\epsilon}^{R-2\epsilon} (a')^2 \, dr
        \leq \epsilon^{1/(n-3)} \lambda_{(k)} (R-\epsilon) + C \epsilon^{(2n^2-9n+11)/2(n-3)}.
    \end{equation}
    By Lemma \ref{lem:fundcs}, for $2\epsilon < r < R - 2\epsilon$, we then deduce that
    \begin{equation}
        |a(r) - a(2\epsilon)|^2 \leq \epsilon^{1/(n-3)} \lambda_{(k)} (R-\epsilon)|R-4\epsilon| + C \epsilon^{(2n^2-9n+11)/2(n-3)}|R-4\epsilon|
    \end{equation}
    which implies that
    \begin{equation}
        a(r) = 1 + O(\epsilon^{1/2(n-3)}) + O(\epsilon^{1/4}).
    \end{equation}
    We therefore obtain
    \begin{equation}\label{2.9}
        \mathcal{R}_k(a) \geq \int_{2\epsilon}^{R-2\epsilon} \lambda_{(k)} (1+o(1))^2 \left(\frac{1}{\epsilon}\right)^{1/2} \, dr
        = \frac{(R-4\epsilon)\lambda_{(k)}}{\epsilon^{1/2}} +o(\epsilon^{-1/2}).
    \end{equation}
    We note that for $\ell > 0$ fixed, Inequality \eqref{2.3} and Inequality \eqref{2.9} hold for any $k \leq \ell +1$ so for all $k \leq \ell$, we deduce that
    \begin{equation}
    \frac{\sigma_{(k+1)}(g_{h_\epsilon})}{\sigma_{(k)}(g_{h_\epsilon})} \geq \frac{\lambda_{(k+1)}}{\lambda_{(k)}} + o(1),
    \end{equation}
    where we used Inequality \eqref{2.3} for $\sigma_{(k)}(g_{h\epsilon})$ in the denominator and Inequality \eqref{2.9} for $\sigma_{(k+1)}(g_{h_\epsilon})$ in the numerator.

    In the case where $n=3$, applying the same arguments as above but with the function
    \begin{equation}\label{2.1b}
        \tilde{h}_\epsilon (r)
        = \begin{cases}
            1, & r \leq \epsilon,\\
            \epsilon^{-1/2}, & 2\epsilon\leq r \leq R-2\epsilon,\\
            R-r & R-\epsilon \leq r \leq R,
        \end{cases}
    \end{equation}
    prove the result.
\end{proof}

In the following example, for $n \geq 4$, we show that the construction used in the proof of Theorem \ref{thm2} is not the only way that the Steklov spectral ratio in this setting can approach the supremum. Roughly, speaking, it is not only metrics for which $h$ is very large that achieve this, but also metrics that are very small.

\begin{ex}\label{eg1}
    Let $n \geq 4$. 
    For $\epsilon$ sufficiently small we define 
    \begin{equation}\label{eqeg1}
        \tilde{h}_\epsilon (r)
        = \begin{cases}
            1, & r \leq \epsilon,\\
            \epsilon^{2}, & \epsilon +\epsilon^2\leq r \leq R-\epsilon^2,\\
            R-r, & R-\epsilon^2 \leq r \leq R,
        \end{cases}
    \end{equation}
    and define $h_\epsilon : [0,R] \to \mathbb{R}$ to be a function that is smooth, decreasing on $[\epsilon, \epsilon + \epsilon^2]$ and equal to $\tilde{h}_\epsilon$ otherwise.
    We claim that
    \begin{equation}\label{eqegconc}
    \frac{\sigma_{(k+1)}(g_{h_\epsilon})}{\sigma_{(k)}(g_{h_\epsilon})} \to \frac{\lambda_{(k+1)}}{\lambda_{(k)}}
    \end{equation}
    as $\epsilon \to 0$.
    
    Taking
    \begin{equation}\label{eqeg2}
        \tilde{a}(r) = 
        \begin{cases}
            1, & 0 \leq r \leq R - \epsilon,\\
            \frac{R-r}{\epsilon}, & R - \epsilon \leq r \leq R,
        \end{cases}
    \end{equation}
    as a test function gives that
    \begin{equation}\label{eqeg3}
        \mathcal{R}(\tilde{a}) \leq 
        \int_0^R \lambda_{(k)} h_\epsilon(r)^{n-3} \, dr + \int_{R-\epsilon}^R \frac{h_\epsilon(r)^{n-1}}{\epsilon^2} \, dr = \lambda_{(k)} \epsilon + O(\epsilon^2).
    \end{equation}
    Hence 
    \begin{equation}\label{eqeg4}
        \sigma_{(k)}(g_{h_\epsilon}) \leq \lambda_{(k)} \epsilon + O(\epsilon^2).
    \end{equation}
    On the other hand, when $a_k$ (with $a_k(0) = 1$) gives rise to an eigenfunction for $\sigma_{(k)}(g_{h_\epsilon})$, we have by Lemma \ref{lem:fundcs} that
    \begin{equation}\label{eqeg5}
        |a_k(r) - 1| = \bigg\vert \int_0^r a_k'(r) \, dr \bigg\vert 
        \leq \left(\int_0^r (a_k')^2\right)^{1/2} r^{1/2}.
    \end{equation}
    So, for $r \leq \epsilon$, we have
    \begin{equation}\label{eqeg6}
        |a_k(r) - 1| \leq \sigma_{(k)}(g_{h_\epsilon})^{1/2}\epsilon^{1/2} \leq \lambda_{(k)}^{1/2} \epsilon + O(\epsilon^{3/2}),
    \end{equation}
    which implies that
    \begin{equation}\label{eqeg7}
        a_k(r) = 1 + O(\epsilon)
    \end{equation}
    for $0 \leq r \leq \epsilon$.
    Hence we obtain that
    \begin{equation}\label{eqeg8}
        \sigma_{(k)}(g_{h_\epsilon}) \geq \int_0^\epsilon a_k^2 \lambda_{(k)} = \lambda_{(k)} \epsilon + O(\epsilon^2).
    \end{equation}
    Therefore by \eqref{eqeg4} and \eqref{eqeg8}, we deduce \eqref{eqegconc}.
\end{ex}

\begin{rem}
    For $n \geq 4$, we remark that the construction used in the proof of Theorem \ref{thm2} also shows that in this setting the Steklov spectral gap $\sigma_{(k+1)}(g_{h_\epsilon}) - \sigma_{(k)}(g_{h_\epsilon}) \to \infty$ as $\epsilon \to 0$. Indeed, by Inequality \eqref{2.8} and Inequality \eqref{2.9} we have that
    \begin{equation*}
        \sigma_{(k+1)}(g_{h_\epsilon}) - \sigma_{(k)}(g_{h_\epsilon}) \geq (\lambda_{(k+1)} - \lambda_{(k)}) \frac{R}{\epsilon^{1/2}} \to \infty, \quad \epsilon \to 0.
    \end{equation*}
    In addition, the construction used in Example \ref{eg1} shows that the Steklov spectral gap
    $\sigma_{(k+1)}(g_{h_\epsilon}) - \sigma_{(k)}(g_{h_\epsilon}) \to 0$ as $\epsilon \to 0$.
    Thus, in order to obtain bounds for the Steklov spectral gap when $n \geq 4$, additional geometric constraints are required. See, for example, \cite{X21} and Theorem \ref{lem:gap}.
\end{rem}

\subsection{Proof of Theorem \ref{thm3}}\label{ss:appsup}
To prove Theorem \ref{thm3} we make use of several lemmas that we introduce below.
\begin{lemma}\label{lem1.3p1}
    Suppose that $a_k$ gives rise to an eigenfunction for $\sigma_{(k)}(g_h)$. If 
\begin{equation}\label{3.4}
    \frac{\sigma_{(k+1)}(g_h)}{\sigma_{(k)}(g_h)} \geq \frac{\lambda_{(k+1)}}{\lambda_{(k)}} - \gamma,
\end{equation}
for $\gamma > 0$, then
\begin{equation}\label{3.8}
    \frac{\int_0^R (a_k')^2h^{n-1} \, dr}{a_k(0)^2 h(0)^{n-1}} 
    \leq \gamma \frac{\sigma_{(k)}(g_h)\lambda_{(k)}}{\lambda_{(k+1)} - \lambda_{(k)}}.
\end{equation}
\end{lemma}

\begin{proof}[Proof of Lemma \ref{lem1.3p1}]
 We recall from the proof of Theorem \ref{thm1} that
    \begin{equation}\label{3.1}
        \sigma_{(k+1)}(g_h)
        \leq \sigma_{(k)}(g_h) + \frac{\lambda_{(k+1)}-\lambda_{(k)}}{\lambda_{(k)}}\frac{\int_0^R \lambda_{(k)} a_k^2h^{n-3} \, dr}{a_k(0)^2 h(0)^{n-1}}.
    \end{equation}
We denote
\begin{equation}\label{3.2}
    \frac{\int_0^R \lambda_{(k)} a_k^2h^{n-3} \, dr}{a_k(0)^2 h(0)^{n-1}} = \psi 
\end{equation}
so that 
\begin{equation}\label{3.3}
    \sigma_{(k)}(g_h) = \frac{\int_0^R (a_k')^2h^{n-1} \, dr}{a_k(0)^2 h(0)^{n-1}} + \psi.
\end{equation}
Then by Inequality \eqref{3.4} and Inequality \eqref{3.1}, we have that
\begin{equation}\label{3.5}
    \frac{\lambda_{(k+1)}}{\lambda_{(k)}}\sigma_{(k)}(g_h) - \gamma\sigma_{(k)}(g_h)
    \leq \sigma_{(k+1)}(g_h) \leq \sigma_{(k)}(g_h) + \frac{\lambda_{(k+1)}-\lambda_{(k)}}{\lambda_{(k)}} \psi
\end{equation}
which implies that
\begin{equation}\label{3.6}
    \psi \geq \sigma_{(k)}(g_h) - \gamma \frac{\sigma_{(k)}(g_h)\lambda_{(k)}}{\lambda_{(k+1)} - \lambda_{(k)}}.
\end{equation}
Hence by \eqref{3.3} we have that
\begin{equation}\label{3.7}
    \psi \geq \frac{\int_0^R (a_k')^2h^{n-1} \, dr}{a_k(0)^2 h(0)^{n-1}} + \psi - \gamma \frac{\sigma_{(k)}(g_h)\lambda_{(k)}}{\lambda_{(k+1)} - \lambda_{(k)}}
\end{equation}
which implies that
\begin{equation*}
    \frac{\int_0^R (a_k')^2h^{n-1} \, dr}{a_k(0)^2 h(0)^{n-1}} 
    \leq \gamma \frac{\sigma_{(k)}(g_h)\lambda_{(k)}}{\lambda_{(k+1)} - \lambda_{(k)}}
\end{equation*}
as required.
\end{proof}

\begin{rem}
    From Inequality \eqref{3.8} we deduce that when $\gamma$ is small,
    \begin{equation}\label{3.9}
        \frac{\int_0^R (a_k')^2h^{n-1} \, dr}{a_k(0)^2 h(0)^{n-1}} 
    \end{equation}
    must also be small.
    The construction that we employed in the proof of Theorem \ref{thm2}, respectively Example \ref{eg1}, ensures that the term in \eqref{3.9} is small by making $h$ very large, respectively small, on a substantial part of $[0,R]$ which leads to $a_k$ being close to a constant. 
\end{rem}

\begin{lemma}\label{lem1.3p2}
    Suppose that $h$ satisfies \eqref{3.10a} and \eqref{3.10b} and that
    \begin{equation*}
    \frac{\sigma_{(k+1)}(g_h)}{\sigma_{(k)}(g_h)} \geq \frac{\lambda_{(k+1)}}{\lambda_{(k)}} - \gamma,
    \end{equation*}
    for some $\gamma > 0$ where
    \begin{equation}\label{gamma1}
        \gamma \leq \frac{1}{4R_1 \rho}
    \end{equation}
    and $\rho$ is a constant depending on $C_1, C_2, R, R_1, \lambda_{(k)}, \lambda_{(k+1)}$ which will be determined below.
    Then, for each $a_k$ that gives rise to an eigenfunction for $\sigma_{(k)}(g_h)$, we have that
    \begin{equation}\label{3.18}
     \frac{1}{2} \leq a_k(r) \leq \frac{3}{2}, \quad \text{ for } 0 < r \leq R_1.
    \end{equation}    
\end{lemma}

\begin{proof}[Proof of Lemma \ref{lem1.3p2}]
    Without loss of generality, we suppose that $a_k(0) = 1$.
    From Inequality \eqref{3.8} and the hypotheses \eqref{3.10a} and \eqref{3.10b} on $h$, we have that
    \begin{equation}\label{3.13}
    \frac{C_1^{n-1} \int_0^{R_1} (a_k')^2 \, dr}{C_2^{n-1}} \leq \frac{\int_0^R (a_k')^2h^{n-1} \, dr}{h(0)^{n-1}} 
    \leq \gamma \frac{\sigma_{(k)}(g_h)\lambda_{(k)}}{\lambda_{(k+1)} - \lambda_{(k)}},
    \end{equation}
    which implies that
    \begin{equation}\label{3.14}
        \int_0^{R_1} (a_k')^2 \, dr \leq \gamma \frac{C_2^{(n-1)}}{C_1^{(n-1)}} \frac{\sigma_{(k)}(g_h)\lambda_{(k)}}{\lambda_{(k+1)} - \lambda_{(k)}}.
    \end{equation}
    We wish to obtain an upper bound independent of $\sigma_{(k)}(g_h)$ so we take 
    \begin{equation*}
        \tilde{a}(r) = 
        \begin{cases}
            1, & 0 \leq r \leq R_1,\\
            \frac{R-r}{R-R_1}, & R_1 \leq r \leq R.
        \end{cases}
    \end{equation*}
    as a test function for $\sigma_{(k)}(g_h)$ to obtain
    \begin{align}\label{3.15}
        \sigma_{(k)}(g_h) \leq \mathcal{R}(\tilde{a})
        &= \frac{\lambda_{(k)}}{h^{n-1}(0)} \int_0^{R_1} h^{n-3}(r) \, dr \nonumber \\
        & \quad + \frac{1}{h^{n-1}(0)} \int_{R_1}^R \left[\left(\frac{1}{R-R_1}\right)^2 h^{n-1}(r) +  \left(\frac{R-r}{R-R_1}\right)^2h^{n-3}(r)\lambda_{(k)} \right] \, dr \nonumber \\
        &\leq \frac{C_2^{n-3}}{C_1^{n-1}}
        \left( R \lambda_{(k)} + \frac{C_2^2}{R- R_1}\right).
    \end{align}
    Hence, we have by \eqref{3.14} and \eqref{3.15} that
    \begin{align}\label{3.16}
        \int_0^{R_1} (a_k')^2 \, dr &
        \leq \gamma \frac{C_2^{2(n-2)}}{C_1^{2(n-1)}} \frac{1}{\lambda_{(k+1)} - \lambda_{(k)}}\left( R \lambda_{(k)}^2 + \frac{C_2^2}{R- R_1}\lambda_{(k)}\right) \nonumber \\
        &= \gamma \rho(C_1, C_2, R, R_1, \lambda_{(k)}, \lambda_{(k+1)}).
    \end{align}
    By Lemma \ref{lem:fundcs} we deduce that for $0 < r \leq R_1$,
    \begin{equation}\label{3.17}
        |a_k(r) - 1|^2 \leq R_1 \int_0^{R_1} (a_k')^2 \, dr \leq R_1 \gamma \rho.
    \end{equation}
    Hence if
    \begin{equation*}
        \gamma \leq \frac{1}{4R_1 \rho},
    \end{equation*}
    then $|a_k(r) - 1|^2 \leq \frac{1}{4}$ which implies that
    \begin{equation*}
     \frac{1}{2} \leq a_k(r) \leq \frac{3}{2}
    \end{equation*}
    for $0 < r \leq R_1$ as required.
\end{proof}

We now employ Lemma \ref{lem1.3p1} and Lemma \ref{lem1.3p2} to show that when $h$ is bounded, it is not possible for the Steklov spectral ratio $\sigma_{(k+1)}(g_h)/\sigma_{(k)}(g_h)$ to approach the supremum in Theorem \ref{thm2}.

\begin{proof}[Proof of Theorem \ref{thm3}]
    The strategy of the proof of Theorem \ref{thm3} is to assume \eqref{3.4} holds and to obtain a contradiction.

    We consider the $a_k$ that gives rise to an eigenfunction for $\sigma_{(k)}(g_h)$. We show that by making a small perturbation of the $a_k$, under the assumption of \eqref{3.4} for suitable $\gamma > 0$ (to be determined below), the Rayleigh quotient corresponding to the perturbed values is smaller than that corresponding to the $a_k$. Since $a_k$ gives rise to an eigenfunction for $\sigma_{(k)}(g_h)$, this gives the desired contradiction.

    We consider the following test function which is a small perturbation of $a_k$:
    \begin{equation}\label{3.19}
        a(r) = 
        \begin{cases}
            a_k(r) - \delta r, & 0 \leq r < \frac{R_1}{2},\\
            a_k(r) - \delta(R_1 - r), & \frac{R_1}{2} \leq r \leq R_1,\\
            a_k(r), & r \geq R_1.
        \end{cases}
    \end{equation}
    The contributions to the Rayleigh quotient, $\mathcal{R}(a)$, on each interval are as follows.
    For $0 \leq r \leq \frac{R_1}{2}$,
    \begin{align}\label{3.20}
        & \frac{\int_0^{R_1/2} \{(a')^2h^{n-1} + \lambda_{(k)} a^2h^{n-3}\} \, dr}{h(0)^{n-1} } \nonumber \\
        & = \frac{\int_0^{R_1/2} \{(a_k' - \delta)^2h^{n-1} + \lambda_{(k)} (a_k -\delta r)^2h^{n-3}\} \, dr}{h(0)^{n-1} } \nonumber \\
        & = \frac{\int_0^{R_1/2} \{(a_k')^2h^{n-1} + \lambda_{(k)} (a_k)^2h^{n-3}\} \, dr}{h(0)^{n-1} } \nonumber \\
        & \quad + \frac{\delta}{h(0)^{n-1} }  \left(\delta \int_0^{R_1/2} \{h^{n-1} + r^2\lambda_{(k)} h^{n-3}\} \, dr
        - 2 \int_0^{R_1/2} (a_k' h^{n-1} + ra_k h^{n-3}\lambda_{(k)}) \, dr \right) \nonumber\\
        & =: \frac{\int_0^{R_1/2} \{(a_k')^2h^{n-1} + \lambda_{(k)} (a_k)^2h^{n-3}\} \, dr}{h(0)^{n-1}} + \frac{T_1}{h(0)^{n-1}} .
    \end{align}
    For $\frac{R_1}{2} \leq r \leq R_1$,
    \begin{align}\label{3.21}
        & \frac{\int_{R_1/2}^{R_1} \{(a')^2h^{n-1} + \lambda_{(k)} a^2h^{n-3}\} \, dr}{h(0)^{n-1}} \nonumber \\
        & = \frac{\int_{R_1/2}^{R_1} \{(a_k' 
        + \delta)^2h^{n-1} + \lambda_{(k)} (a_k -\delta(R_1 - r))^2h^{n-3}\} \, dr}{h(0)^{n-1} } \nonumber \\
        & = \frac{\int_{R_1/2}^{R_1} \{(a_k')^2h^{n-1} + \lambda_{(k)} (a_k)^2h^{n-3}\} \, dr}{h(0)^{n-1} } \nonumber \\
        & \quad + \frac{\delta}{h(0)^{n-1}}  \left(\delta \int_{R_1/2}^{R_1} \{h^{n-1} + (R_1 - r)^2\lambda_{(k)} h^{n-3}\} \, dr
        + 2 \int_{R_1/2}^{R_1} (a_k' h^{n-1} - (R_1 -r) a_k h^{n-3}\lambda_{(k)}) \, dr \right) \nonumber\\
        & = \frac{\int_{R_1/2}^{R_1} \{(a_k')^2h^{n-1} + \lambda_{(k)} (a_k)^2h^{n-3}\} \, dr}{h(0)^{n-1}} + \frac{T_2}{h(0)^{n-1}} .
    \end{align}
    So we see that
    \begin{equation}\label{3.22}
        \mathcal{R}(a) = \mathcal{R}(a_k) + \frac{T_1}{h(0)^{n-1}}  + \frac{T_2}{h(0)^{n-1}} .
    \end{equation}
    In order to show that $\mathcal{R}(a) < \mathcal{R}(a_k)$, we show that for certain $\gamma$, $T_1 < 0$ and $T_2 < 0$. We observe that both $T_1$ and $T_2$ are of the form
    \begin{equation}\label{3.23}
        \delta(\delta A - B)
    \end{equation}
    and if $A, B > 0$, then
    \begin{equation}\label{3.24}
        \delta(\delta A - B) < 0 \iff \delta < \frac{B}{A}.
    \end{equation}
    We have that
    \begin{equation}\label{3.25}
    T_1 = \delta(\delta A_1 - B_1) = 
        \delta\left(\delta \int_0^{R_1/2} \{h^{n-1} + r^2\lambda_{(k)} h^{n-3}\} \, dr
        - 2 \int_0^{R_1/2} (a_k' h^{n-1} + ra_k h^{n-3}\lambda_{(k)}) \, dr \right)
    \end{equation}
    So
    \begin{equation}\label{3.26}
        \frac{B_1}{A_1} = \frac{2 \int_0^{R_1/2} (a_k' h^{n-1} + ra_k h^{n-3}\lambda_{(k)}) \, dr }{\int_0^{R_1/2} \{h^{n-1} + r^2\lambda_{(k)} h^{n-3}\} \, dr}.
    \end{equation}
    We see immediately that $A_1 \geq 0$. We can also ensure $B_1 \geq 0$ by imposing constraints on $\gamma$ as follows.

    We observe that
    \begin{equation}\label{3.27}
        a_k' h^{n-1} + ra_k h^{n-3}\lambda_{(k)} \geq ra_k h^{n-3}\lambda_{(k)} - |a_k'| h^{n-1}.
    \end{equation}
    Now we have that
    \begin{align}\label{3.28}
        \bigg\vert \int_{0}^{R_1/2} a_k' h^{n-1} \, dr \bigg\vert &\leq
        \int_0^{R_1/2} |a_k'| h^{n-1} \, dr \nonumber \\
        & \leq C_2^{n-1} \int_0^{R_1/2} |a_k'| \, dr \nonumber \\
        & \leq C_2^{n-1} \left(\int_0^{R_1/2} |a_k'|^2 \, dr \right)^{1/2} \left(\frac{R_1}{2}\right)^{1/2} \nonumber \\
        & \leq C_2^{n-1} \left(\frac{R_1}{2}\right)^{1/2}  \gamma^{1/2} \rho^{1/2},
    \end{align}
    where we used the Cauchy--Schwarz Inequality and then Inequality \eqref{3.16}. In addition, we have that
    \begin{equation}\label{3.29}
        \int_0^{R_1/2} ra_k h^{n-3}\lambda_{(k)} \, dr
        \geq \frac{C_1^{n-3}}{2}\lambda_{(k)} \int_0^{R_1/2} r \, dr
        = \frac{C_1^{n-3}R_1^2}{16} \lambda_{(k)}
    \end{equation}
    by Inequality \eqref{3.18}.
    Hence we have that
    \begin{equation}\label{3.30}
        \int_0^{R_1/2} (a_k' h^{n-1} + ra_k h^{n-3}\lambda_{(k)}) \, dr
        \geq \frac{C_1^{n-3}R_1^2}{16} \lambda_{(k)} - C_2^{n-1} \left(\frac{R_1}{2}\right)^{1/2}  \gamma^{1/2} \rho^{1/2}. 
    \end{equation}
    The right-hand side of Inequality \eqref{3.30} is non-negative if and only if
    \begin{gather}\label{3.31}
        C_2^{n-1} \left(\frac{R_1}{2}\right)^{1/2}  \gamma^{1/2} \rho^{1/2} \leq \frac{C_1^{n-3}R_1^2}{16} \lambda_{(k)} \\ \iff
        \gamma \leq \frac{C_1^{2(n-3)}}{C_2^{2(n-1)}} \frac{R_1^3}{128} \frac{\lambda_{(k)}^2}{\rho}
        = \frac{C_1^{4(n-2)}}{C_2^{2(2n-3)}} \frac{R_1^3}{128} (\lambda_{(k+1)} - \lambda_{(k)}) \left( R + \frac{C_2^2}{(R- R_1)\lambda_{(k)}}\right) ^{-1}.
    \end{gather}
    For such values of $\gamma$ and $\delta \leq \frac{B_1}{A_1}$, we have that $T_1 < 0$. 
    
    By performing the analogous calculations for $T_2$, we obtain the same upper bound for $\gamma$ as in Inequality \eqref{3.31}. Hence, for such values of $\gamma$ and $\delta \leq \frac{B_2}{A_2}$, we have that $T_2 < 0$.

    Therefore, for 
    \begin{equation}
    \begin{split}
        \gamma &= \min\left\{\frac{1}{4R_1} \frac{C_1^{2(n-1)}}{C_2^{2(n-2)}}(\lambda_{(k+1)} - \lambda_{(k)})\left( R \lambda_{(k)}^2 + \frac{C_2^2}{R- R_1}\lambda_{(k)}\right)^{-1}, \right. \\ 
        & \left. \quad \quad \quad \frac{C_1^{4(n-2)}}{C_2^{2(2n-3)}} \frac{R_1^3}{128} (\lambda_{(k+1)} - \lambda_{(k)}) \left( R + \frac{C_2^2}{(R- R_1)\lambda_{(k)}}\right)^{-1}\right\}, \label{3.32}
    \end{split}
    \end{equation}
    and $\delta \leq \min\{\frac{B_1}{A_1}, \frac{B_2}{A_2}\}$, we have that $\mathcal{R}(a) < \mathcal{R}(a_k)$ which is a contradiction. Note that the first condition in \eqref{3.32} comes from \eqref{gamma1}.
\end{proof}

\subsection{Proof of Theorem \ref{thm:revolution}}\label{ss:upbound3D}
In this section, we prove Theorem \ref{thm:revolution}.

\begin{proof}[Proof of Theorem \ref{thm:revolution}]
    We recall that $h(R) = 0$ and $h'(R)=-1$. Thus, there exists $\rho > 0$ such that if $r \in [R-\rho, R]$, we have
    \begin{equation}\label{3.37}
        \frac{1}{2} (R-r) \leq h(r) \leq 2(R-r).
    \end{equation}
    Let $0 < \epsilon <\rho$. Similarly to the proof of Theorem \ref{thm2}, we take
    \begin{equation}\label{3.38}
        \tilde{a}(r) = 
        \begin{cases}
            1, & 0 \leq r \leq R - \epsilon,\\
            \frac{R-r}{\epsilon}, & R - \epsilon \leq r \leq R,
        \end{cases}
    \end{equation}
    as a test function and make use of the upper bound in \eqref{3.37} to obtain that
    \begin{align}
        \sigma_{(k)}(g_h) 
        &\leq \frac{1}{h(0)^2} \left( \lambda_{(k)} (R-\epsilon) + \int_{R-\epsilon}^R
        \left[ \frac{h^2}{\epsilon^2} + \lambda_{(k)} \left(\frac{R-r}{\epsilon}\right)^2\right] \, dr \right) \nonumber\\
        &\leq \frac{1}{h(0)^2} ( \lambda_{(k)} (R-\epsilon) + (4 + \lambda_{(k)})\epsilon).
    \end{align}
    Then, by taking the limit as $\epsilon \to 0$, we obtain that
    \begin{equation*}
        \sigma_{(k)}(g_h) \leq \frac{R \lambda_{(k)}}{h(0)^2}.
    \end{equation*}

    To prove that the previous inequality is strict, we assume that there exists a $h \in C^\infty([0,R])$ such that $h(R) = 0$, $h'(R) = -1$ and 
    \begin{equation*}
        \sigma_{(k)}(g_h) = \frac{R \lambda_{(k)}}{h(0)^2}
    \end{equation*}
    and obtain a contradiction.
    Given such a $h$, it is possible to construct a function $\overline{h} \in C^\infty([0,R])$ such that $\overline{h}(R) = 0$, $\overline{h}'(R) = -1$, $\overline{h}(0) = h(0)$, $\overline{h}(r) > h(r)$ for $r \in [\frac{R}{4}, \frac{R}{2}]$ and $\overline{h}(r) \geq h(r)$ for $r \in [0,R] \setminus [\frac{R}{4}, \frac{R}{2}]$. Let $\overline{a}_k$ be a function that gives rise to an eigenfunction corresponding to $\sigma_k(g_{\overline{h}})$. Then $\overline{a}_k$ is not a constant function since constant functions do not satisfy \eqref{eq:defa} for $k \geq 1$.
    Taking $\overline{a}_k$ as a test function for $\sigma_k(g_{h})$, we obtain that
    \begin{equation*}
       \sigma_k(g_{h}) \leq \frac{\int_0^R \{(\overline{a}_k')^2h^2 + \lambda_{(k)} \overline{a}_k^2\} \, dr}{\overline{a}_k(0)^2 h(0)^2}
       < \frac{\int_0^R \{(\overline{a}_k')^2\overline{h}^2 + \lambda_{(k)} \overline{a}_k^2\} \, dr}{\overline{a}_k(0)^2 \overline{h}(0)^2 } = \sigma_k(g_{\overline{h}}) \leq \frac{R \lambda_{(k)}}{h(0)^2},
    \end{equation*}
    which is a contradiction.

    To show that $\sup\{\sigma_{(k)}(g_h):h(0)=h_0\}=\frac{R\lambda_{(k)}}{h_0^2}$, we follow the same arguments as in the proof of Theorem \ref{thm2} with the function
        \begin{equation}\label{3.39}
        \tilde{h}_\epsilon (r)
        = \begin{cases}
            h_0, & r \leq \epsilon,\\
            h_0\epsilon^{-1/2}, & 2\epsilon\leq r \leq R-2\epsilon,\\
            R-r & R-\epsilon \leq r \leq R,
        \end{cases}
    \end{equation}
    to obtain, analogously to \eqref{2.9}, that
    \begin{equation*}
        \sigma_{(k)}(g_{h_\epsilon}) \geq \frac{1}{h_0^2} \int_{2\epsilon}^{R-2\epsilon} \lambda_{(k)} (1 + O(\epsilon^{1/2}))^2 \, dr = \frac{|R - 4\epsilon| \lambda_{(k)}}{h_0^2} + O(\epsilon^{1/2}).
    \end{equation*}
    Taking the limit as $\epsilon \to 0$ concludes the proof.
\end{proof}

\subsection{Upper bounds for Steklov spectral gaps}\label{ss:gap}
In this section, we prove Theorem \ref{thm:gaprevolution} and Theorem \ref{lem:gap}.
Both proofs make use of arguments from the proof of Theorem \ref{thm1}. The former also employs the upper bound from Theorem \ref{thm:revolution}, while the latter exploits the additional hypotheses that $h$ is bounded.

\begin{proof}[Proof of Theorem \ref{thm:gaprevolution}]
    As in the proof of Theorem \ref{thm1}, we take a function $a_k$ that gives rise to an eigenfunction for $\sigma_{(k)}(g_h)$ and use it as a test function in the Rayleigh quotient corresponding to $\sigma_{(k+1)}(g_h)$. By \eqref{eq:gap0} and the fact that $a_k$ is not a constant function, we have that
    \begin{equation}\label{3.35}
        \sigma_{(k+1)}(g_h)
        < \sigma_{(k)}(g_h) + \frac{\lambda_{(k+1)}-\lambda_{(k)}}{\lambda_{(k)}}\sigma_{(k)}(g_h).
    \end{equation}
    Therefore, by Theorem \ref{thm:revolution} we have that
    \begin{equation}\label{3.36}
        \sigma_{(k+1)}(g_h)
        < \sigma_{(k)}(g_h) + \frac{R(\lambda_{(k+1)}-\lambda_{(k)})}{h(0)^2},
    \end{equation}
    which implies
    \begin{equation*}
         \sigma_{(k+1)}(g_h) - \sigma_{(k)}(g_h)<\frac{R(\lambda_{(k+1)}-\lambda_{(k)})}{h(0)^2}
    \end{equation*}
    as required.

    Moreover, this upper bound is optimal. Indeed, consider the family of smooth functions $(h_\epsilon)$ constructed in the proof of Theorem \ref{thm:revolution}. By the previous inequality, for $k\geq0$, we have that
    \begin{equation*}
        \sigma_{(k+1)}(g_{h_\epsilon})=\sum\limits_{j=0}^{k}\sigma_{(k+1-j)}(g_{h_\epsilon})-\sigma_{(k-j)}(g_{h_\epsilon})\leq\sum\limits_{j=0}^k\frac{R(\lambda_{(k+1-j)}-\lambda_{(k-j)})}{h_0^2}=\frac{R\lambda_{(k+1)}}{h_0^2}.
    \end{equation*}
    By Theorem \ref{thm:revolution}, we have that $\sigma_{(k+1)}(g_{h_\epsilon})\to\frac{R\lambda_{(k+1)}}{h_0^2}$ as $\epsilon\to0$. This implies that each term in the previous sum converges, namely, for all $0\leq j\leq k$, we have that
    \begin{equation*}
        \sigma_{(k+1-j)}(g_{h_\epsilon})-\sigma_{(k-j)}(g_{h_\epsilon}) \to \frac{R(\lambda_{(k+1-j)}-\lambda_{(k-j)})}{h_0^2},
    \end{equation*}
    as $\epsilon\to0$.
\end{proof}

Finally we prove Theorem \ref{lem:gap}.

\begin{proof}[Proof of Theorem \ref{lem:gap}]
    As in the proof of Theorem \ref{thm:revolution}, we take $\tilde{a}$ as defined in \eqref{3.38} 
    as a test function and employ the upper bound in \eqref{3.37} and the bounds on $h$ given in the statement of Theorem \ref{lem:gap} to obtain that
    \begin{equation}\label{3.33}
        \sigma_{(k)}(g_h) \leq 
    \frac{\int_0^R \{(\tilde{a}')^2h^{n-1} + \lambda_{(k)} \tilde{a}^2 h^{n-3}\} \, dr}{h(0)^{n-1} }
    \leq \frac{C_2^{n-3}}{h(0)^{n-1}}
        \left( R \lambda_{(k)} + \frac{2}{3}\epsilon\right).
    \end{equation}
    Now by \eqref{eq:gap0} and \eqref{3.33}, we have that
    \begin{align}\label{3.34}
        \sigma_{(k+1)}(g_h)
        &\leq \sigma_{(k)}(g_h) + \frac{\lambda_{(k+1)}-\lambda_{(k)}}{\lambda_{(k)}}\sigma_{(k)}(g_h) \nonumber\\
        &\leq \sigma_{(k)}(g_h) + \frac{(\lambda_{(k+1)}-\lambda_{(k)}) C_2^{n-3} R}{h(0)^{n-1}} + 
        \frac{\lambda_{(k+1)}-\lambda_{(k)}}{\lambda_{(k)}}
        \frac{2C_2^{n-3} \epsilon }{3 h(0)^{n-1}}.
    \end{align}
    Then, letting $\epsilon \to 0$, we obtain \eqref{eq:gap1} as required.
\end{proof}

\bibliographystyle{plain}
\bibliography{biblio}

\end{document}